\documentclass[a4paper,12pt]{amsart}

\usepackage{amsmath}
\usepackage{amssymb}
\usepackage{mathrsfs}
\usepackage{enumerate}
\usepackage{ifthen}
\usepackage{graphicx}
\usepackage[T1]{fontenc} 

\nonstopmode \numberwithin{equation}{section}
\newtheorem{thm}{Theorem}[section]
\newtheorem{cor}{Corollary}[section]
\newtheorem{lem}{Lemma}[section]
\newtheorem{prop}{Proposition}[section]

\newtheorem{conj}{Conjecture}

\theoremstyle{definition}

\newtheorem{prob}{Problem}[section]
\newtheorem{rem}{Remark}[section]

\newcounter{minutes}
\divide\time by 60
\newcounter{hours}
\multiply\time by 60
\addtocounter{minutes}{-\time}

\newcounter {own}
\def\theown {\thesection       .\arabic{own}}

\newenvironment{pf}[1][]{%
 \vskip 3mm
 \noindent
 \ifthenelse{\equal{#1}{}}%
  {{\slshape Proof. }}%
  {{\slshape #1.} }%
 }%
{\qed\bigskip}

\newcounter{alphabet}
\newcounter{tmp}

\def\be{\begin{equation}}
\def\ee{\end{equation}}

\newcommand{\bee}{\begin{enumerate}}
\newcommand{\eee}{\end{enumerate}}

\newcommand{\blem}{\begin{lem}}
\newcommand{\elem}{\end{lem}}
\newcommand{\bthm}{\begin{thm}}
\newcommand{\ethm}{\end{thm}}
\newcommand{\bcor}{\begin{cor}}
\newcommand{\ecor}{\end{cor}}
\newcommand{\beg}{\begin{examp}}
\newcommand{\eeg}{\end{examp}}
\newcommand{\begs}{\begin{examples}}
\newcommand{\eegs}{\end{examples}}
\newcommand{\bdefe}{\begin{defin}}
\newcommand{\edefe}{\end{defin}}
\newcommand{\bprob}{\begin{prob}}
\newcommand{\eprob}{\end{prob}}
\newcommand{\bei}{\begin{itemize}}
\newcommand{\eei}{\end{itemize}}

\newcommand{\bcon}{\begin{conj}}
\newcommand{\econ}{\end{conj}}
\newcommand{\bcons}{\begin{conjs}}
\newcommand{\econs}{\end{conjs}}
\newcommand{\bprop}{\begin{prop}}
\newcommand{\eprop}{\end{prop}}
\newcommand{\br}{\begin{rem}}
\newcommand{\er}{\end{rem}}
\newcommand{\brs}{\begin{rems}}
\newcommand{\ers}{\end{rems}}
\newcommand{\bo}{\begin{obser}}
\newcommand{\eo}{\end{obser}}
\newcommand{\bos}{\begin{obsers}}
\newcommand{\eos}{\end{obsers}}
\newcommand{\bpf}{\begin{pf}}
\newcommand{\epf}{\end{pf}}
\newcommand{\ba}{\begin{array}}
\newcommand{\ea}{\end{array}}
\newcommand{\beq}{\begin{eqnarray}}
\newcommand{\beqq}{\begin{eqnarray*}}
\newcommand{\eeq}{\end{eqnarray}}
\newcommand{\eeqq}{\end{eqnarray*}}

\begin{document}

\title{On a class of univalent functions defined by a differential inequality}

\author{Md Firoz Ali}
\address{Md Firoz Ali,
Department of Mathematics,
National Institute of Technology Calicut,
Calicut- 673601, Kerala, India.}
\email{ali.firoz89@gmail.com, firozali@nitc.ac.in}

\author{Vasudevarao Allu}
\address{Vasudevarao Allu,
Discipline of Mathematics,
School of Basic Sciences,
Indian Institute of Technology  Bhubaneswar,
Argul, Bhubaneswar, PIN-752050, Odisha (State),  India.}
\email{avrao@iitbbs.ac.in}

\author{Hiroshi Yanagihara}
\address{Department of Applied Science,
	Faculty of Engineering,
	Yamaguchi University,
	Tokiwadai, Ube 755,
	Japan}
\email{hiroshi@yamaguchi-u.ac.jp}


\def\thefootnote{}
\footnotetext{ {\tiny File:~\jobname.tex,
printed: \number\year-\number\month-\number\day,
          \thehours.\ifnum\theminutes<10{0}\fi\theminutes }
} \makeatletter\def\thefootnote{\@arabic\c@footnote}\makeatother

\begin{abstract}
For $0<\lambda\le 1$, let $\mathcal{U}(\lambda)$ be the class analytic functions $f(z)= z+\sum_{n=2}^{\infty}a_n z^n$ in the unit disk $\mathbb{D}$ satisfying $|f'(z)(z/f(z))^2-1|<\lambda$ and $\mathcal{U}:=\mathcal{U}(1)$. In the present article, we prove that the class $\mathcal{U}$ is contained in the closed convex hull of the class of starlike functions and using this fact, we solve some extremal problems such as integral mean problem and arc length problem for functions in $\mathcal{U}$. By means of the so called theory of star functions, we also solve the integral mean problem for functions in $\mathcal{U}(\lambda)$. We also obtain the estimate of the Fekete-Szeg\"{o} functional and the pre-Schwarzian norm of certain nonlinear integral transform of functions in $\mathcal{U}(\lambda)$. Further, for the class of meromorphic functions which are defined in $\Delta:=\{\zeta\in\mathbb{\widehat{C}}:|\zeta|>1\}$ and associated with the class $\mathcal{U}(\lambda)$, we obtain a sufficient condition for a function $g$ to be an extreme point of this class.

\end{abstract}

\thanks{}

\maketitle
\pagestyle{myheadings}
\markboth{Md Firoz Ali,  Vasudevarao Allu and Hiroshi Yanagihara}{On a class of  univalent functions defined by a differential inequality}

\section{Introduction}
We denote the complex plane by $\mathbb{C}$ and the extended complex plane by $\mathbb{\widehat{C}}=\mathbb{C}\cup\{\infty\}$. Let $\mathcal{H}$ denote the class of analytic functions in the unit disk $\mathbb{D}:=\{z\in\mathbb{C}:\, |z|<1\}$. Then $\mathcal{H}$ is a locally convex topological vector space endowed with the topology of uniform convergence over compact subsets of $\mathbb{D}$. Denote by $\mathcal{A}$ the subclass of $\mathcal{H}$ of functions $f$ with Taylor series
\begin{equation}\label{p9-001}
f(z)= z+\sum_{n=2}^{\infty}a_n z^n.
\end{equation}
The subclass $\mathcal{S}$ of $\mathcal{A}$, consisting of univalent (one-to-one) functions  has attracted much interest for over a century, and is a central area of research in the theory of complex analysis.

Although the class $\mathcal{S}$ is the main attraction, various geometric subclasses (e.g. starlike, convex and close-to-convex) have been extensively studied,  some of which appear  naturally in  different areas in the  theory of quasiconformal mappings. A function $f\in\mathcal{A}$ is called starlike (respectively convex), if $f(\mathbb{D})$ is starlike with respect to the origin (respectively convex). Let $\mathcal{S}^*$ and $\mathcal{C}$ denote the class of starlike and convex functions in $\mathcal{S}$, respectively. It is well-known that a function $f\in\mathcal{A}$ belongs to $\mathcal{S}^*$ if, and only, if ${\rm Re\,}\left(zf'(z)/f(z)\right)>0$ for $z\in\mathbb{D}$. Similarly, a function $f\in\mathcal{A}$ belongs to $\mathcal{C}$ if, and only if, ${\rm Re\,}\left(1+(zf''(z)/f'(z))\right)>0$ for $z\in\mathbb{D}$. From the above, it is easy to see that $f\in\mathcal{S}^*$ if, and only if, $J[f]\in\mathcal{C}$, where $J[f]$ denotes the Alexander transform of $f\in\mathcal{A}$ defined by
$$
J[f](z):=\int_0^z \frac{f(\xi)}{\xi}\,d\xi=\int_0^1 \frac{f(tz)}{t}\,dt.
$$
In 1960, Biernacki claimed that $f\in\mathcal{S}$ implies $J[f]\in\mathcal{S}$, but this turned out to be wrong (see \cite[Theorem 8.11]{Duren-book}). Later, Kim and Merkes \cite{Kim-Merkes-1972} have considered the nonlinear integral transform $J_\alpha$ defined by
\begin{equation}\label{p9-005}
J_\alpha[f](z):=\int_0^z \left(\frac{f(\xi)}{\xi}\right)^{\alpha}\,d\xi
\end{equation}
for complex numbers $\alpha$ and for the functions $f$ in the class
$$
\mathcal{ZF}=\{f\in\mathcal{A}: f(z)\ne 0\quad\text{for all}\quad 0<|z|<1\}
$$
and proved that $J_\alpha(\mathcal{S}):=\{J_\alpha[f]:f\in\mathcal{S}\} \subset \mathcal{S}$ for $|\alpha|\le 1/4$. For starlike functions, it is known that $J_\alpha(\mathcal{S}^*)\subset \mathcal{S}$ when either $|\alpha|\le 1/2$ or $\alpha\in[1/2,3/2]$ (see \cite{Aksentev-Nezhmetdinov-1982,Singh-Chichra-1977}).

For $0<\lambda\le 1$, let $\mathcal{U}(\lambda)$ be the class of functions $f\in\mathcal{A}$ satisfying
$$
\left|f'(z)\left(\frac{z}{f(z)}\right)^2-1 \right|<\lambda \quad\mbox{ for } z\in\mathbb{D}.
$$
Since $f'(z)(z/f(z))^2\ne 0$ in $\mathbb{D}$, it follows that every $f\in\mathcal{U}(\lambda)$ is non-vanishing in $\mathbb{D}\setminus \{0\}$. We set $\mathcal{U}:=\mathcal{U}(1)$. It is also clear that every $f\in\mathcal{U}(\lambda)$ is locally univalent. Moreover, it is well-known that every $f\in\mathcal{U}$ is univalent in $\mathbb{D}$, i.e., $\mathcal{U}\subsetneq\mathcal{S}$ (see \cite{Aksentev-1958}) and hence, for $0\le \lambda<1$, one has $\mathcal{U}(\lambda)\subsetneq\mathcal{S}$. On the other hand, it follows from \cite{Fournier-Ponnusamy-2007,Obradovic-Ponnusamy-2001} that neither $\mathcal{U}$ is included in $\mathcal{S}^*$ nor $\mathcal{S}^*$ is included in $\mathcal{U}$. However, for $0<\lambda<1/\sqrt{2}$, functions in $\mathcal{U}(\lambda)$ are starlike with the with the additional assumption that $f''(0)=0$, . In 1977, Singh \cite{Singh-1977} obtained an estimate for the radius of starlikeness of the class $\mathcal{U}$ which is surprisingly close to unity.

\bigskip
One of the important notions in geometric function theory is subordination. Let $f$ and $g$ be two analytic functions in $\mathbb{D}$. We say that $f$ is subordinate to $g$, written as $f \prec g$ or $f(z) \prec g(z)$, if there exists an analytic function $\omega: \mathbb{D} \rightarrow \mathbb{D}$ with $\omega(0)=0$ such that $f(z)= g(\omega(z))$ for $z\in\mathbb{D}$. In the case, $g$ is univalent in $\mathbb{D}$, subordination is equivalent to $f(0)= g(0)$ and $f(\mathbb{D})\subseteq g(\mathbb{D})$.

In 1995, Obradovi\'{c} \cite{Obradovic-1995} proved that if $f\in\mathcal{U}$ then $f(z)/z\prec k(z)/z$, where $k(z)=z/(1-z)^2$ is the Koebe function. Recently, this result has been generalized for the class $\mathcal{U}(\lambda)$ by Obradovi\'{c} {\it et al.} \cite{Obradovic-Ponnusamy-Wirths-2016}, where it has been proved that if $f\in\mathcal{U}(\lambda)$ then
\begin{equation}\label{p9-010}
\frac{f(z)}{z}\prec \frac{k_{\lambda}(z)}{z}\quad\mbox{with}\quad k_{\lambda}(z)=\frac{z}{(1-z)(1-\lambda z)}.
\end{equation}
Clearly, $k_1(z)=k(z)$ is the well-known Keobe function.

Suppose  $X$ is a linear topological space and $U\subseteq X$.
The closed convex hull of $U$, denoted by $\overline{co}\,U$, is defined as the intersection of all closed convex sets containing $U$.
For $U\subseteq V\subseteq X$, we say that $U$ is an extremal subset of $V$ if $u=tx+(1-t)y$, where $u\in U$, $x,y\in V$ and $0<t<1$ then $x$ and $y$ both belong to $U$. An extremal subset of $U$ consisting of just one point is called an extreme point of $U$. We denote  the set of all extreme points of $U$ by $E(U)$. For a general reference and for many important results on this topic, we refer to \cite{Hallenbeck-MacGregor-1984}.

\bigskip
In the present article, among other results, we prove that the class $\mathcal{U}$ is contained in the closed convex hull of $\mathcal{S}^*$, i.e., $\mathcal{U}\subsetneq\overline{co}\,\mathcal{S}^*$ and using this fact, we solve some extremal problems such as integral mean problem and arc length problem for functions in $\mathcal{U}$. We also consider the nonlinear integral transform $J_\alpha$ defined by (\ref{p9-005}) for functions in $\mathcal{U}(\lambda)$ and obtain its pre-Schwarzian norm estimates. We also discuss the Fekete-Szeg\"{o} problem for functions in $\mathcal{U}(\lambda)$. Finally, we consider the class $\mathcal{M}_{0}(\lambda)$ of meromorphic functions defined in $\Delta:=\{\zeta\in\mathbb{\widehat{C}}:|\zeta|>1\}$ and associated with the class $\mathcal{U}(\lambda)$. For the class $\mathcal{M}_{0}(\lambda)$, we obtain a sufficient condition for a function $g$ to be an extreme point of the class $\mathcal{M}_{0}(\lambda)$.

\section{Properties of functions in $\mathcal{U}(\lambda)$}\label{Main Results-1}

We need the following results to prove our main results.

\begin{prop}\cite{Muhanna-Hallenbeck-1992}\label{p9-proposition-001}
Let $F$ be an analytic and univalent function in the unit disk $\mathbb{D}$. If $F(z)\ne0$ in $\mathbb{D}$ and $\mathbb{C}\setminus F(\mathbb{D})$ is convex domain, then any analytic function $f$ in $s(F^n):=\{g: g\prec F^n\}$, $n=1,2,\ldots$, can be expressed as
$$
f(z)=\int_{|x|=1} F^n(xz)\,d\mu(x),
$$
for some probability measure $\mu$ on the unit circle $|x|=1$.
\end{prop}

\begin{prop}\cite[Theorem 5.8]{Hallenbeck-MacGregor-1984} \label{p9-proposition-005}
The set $\overline{co}\,\mathcal{S}^*$ consists of all functions represented by
$$
f(z)=\int_{|x|=1} \frac{z}{(1-xz)^2}\,d\mu(x),
$$
where $\mu$ varies over the set of all probability measures on the unit circle $|x|=1$. Also,
$$
E(\overline{co}\,\mathcal{S}^*)=\left\{\frac{z}{(1-xz)^2}:|x|=1\right\}.
$$
\end{prop}

Using Propositions \ref{p9-proposition-001} and \ref{p9-proposition-005}, we prove the following integral representation for functions in $\mathcal{U}$.

\begin{thm}\label{p9-theorem-001}
Each function $f\in\mathcal{U}$ has an integral representation of the form
\begin{equation}\label{p9-015}
f(z)=\int_{|x|=1} \frac{z}{(1-xz)^2}\,d\mu(x)
\end{equation}
for some probability measure $\mu$ on the unit circle $|x|=1$. Moreover, $\mathcal{U}\subsetneq\overline{co}\,\mathcal{S}^*$.
\end{thm}

\begin{proof}
If $f\in\mathcal{U}$ then (\ref{p9-010}) holds, i.e., $f(z)/z\prec k(z)/z=(1-z)^{-2}$. We note that the function $k(z)/z=(1-z)^{-2}$ is analytic and univalent in $\mathbb{D}$ and maps the unit disk $\mathbb{D}$ onto a domain $\Omega$ whose complement is convex. Indeed, the function $k(z)/z$ maps the unit  circle $\partial\mathbb{D}$ onto a hyperbola given by the equation $v^2=-(u-1/4)$. In view of Proposition \ref{p9-proposition-001}, $f(z)/z$ can be expressed as
$$
\frac{f(z)}{z}=\int_{|x|=1} \frac{k(xz)}{xz}\,d\mu(x),
$$
which is equivalent to (\ref{p9-015}). From the representation (\ref{p9-015}) and Proposition \ref{p9-proposition-005} it follows that $\mathcal{U}\subsetneq\overline{co}\,\mathcal{S}^*$. The proper inclusion follows from the fact that $\mathcal{U}$ is not included in $\mathcal{S}^*$ and $\mathcal{S}^*\subsetneq\overline{co}\,\mathcal{S}^*$.
\end{proof}

\begin{rem}
It follows from Theorem \ref{p9-theorem-001} that every function $f\in\mathcal{U}$ has the integral representation of the form (\ref{p9-015}). But unfortunately, we are unable to conclude similar integral representation for functions in $\mathcal{U}(\lambda)$, $0<\lambda<1$, using the subordination relation $f(z)/z\prec k_{\lambda}(z)/z$, because the function $\phi_{\lambda}(z):=k_{\lambda}(z)/z$, $0<\lambda<1$, maps $\mathbb{D}$ onto a domain $\phi_{\lambda}(\mathbb{D})$ whose complement is not convex. Indeed, if we let $\phi_{\lambda}(e^{i\theta})=u(\theta)+iv(\theta)$, then a simple computation shows that
\begin{equation}\label{p9-018}
u(\theta)=\frac{1-\lambda-2\lambda\cos\theta}{2(1+\lambda^2-2\lambda\cos\theta)}\quad\mbox{ and }\quad
v(\theta)= \frac{(1+\lambda-2\lambda\cos\theta)\cot(\theta/2)}{2(1+\lambda^2-2\lambda\cos\theta)},
\end{equation}
and so, the equation of the curve $\phi_{\lambda}(e^{i\theta})$, $0\le\theta<2\pi$, is given by
\begin{equation}\label{p9-020}
v^2=-\frac{(1+(\lambda-1)u)^2(2(\lambda+1)u-1)}{(\lambda+1)(2(\lambda-1)^2u+3\lambda-1)}.
\end{equation}
From (\ref{p9-018}), we note that for $0<\lambda<1$, $|v|\rightarrow +\infty$ and $u\rightarrow (1-3\lambda)/(2(1-\lambda)^2)$ as $\theta\rightarrow 0$, and so $u= (1-3\lambda)/(2(1-\lambda)^2)$ is an  vertical asymptote of the curve (\ref{p9-020}). Hence, $\mathbb{C}\setminus\phi_{\lambda}(\mathbb{D})$ is not convex.
\end{rem}

%

It is important to note that by using the integral representation (\ref{p9-015}) for functions in $\mathcal{U}$, one can easily obtain the following coefficient estimates for functions in $\mathcal{U}$.

\begin{cor}\label{p9-corollary-005}
If $f\in\mathcal{U}$ is of the form (\ref{p9-001}) then $|a_n|\le n$. Equality holds if, and only if, $f$ is a rotation of the Koebe function.
\end{cor}

If $\mathscr{G}$ is a convex subset of $\mathcal{H}$ and $J:\mathcal{H}\rightarrow \mathbb{R}$ then $J$ is called convex on $\mathscr{G}$ if $J(tf+(1-t)g)\le tJ(f)+(1-t)J(g)$ whenever $f,g\in\mathscr{G}$ and $0\le t\le 1$. We note that $\overline{co}\,\mathcal{S}^*$ is a compact subset of $\mathcal{H}$ and $E(\overline{co}\,\mathcal{S}^*)\subsetneq \mathcal{U}$. Hence, for any real-valued, continuous and convex functional $J$ on $\overline{co}\,\mathcal{S}^*$, by \cite[Theorems 4.5 and 4.6]{Hallenbeck-MacGregor-1984}, we have
$$
\max_{f\in\mathcal{U}} J(f)\le \max_{f\in\overline{co}\,\mathcal{S}^*} J(f)=\max_{f\in E(\overline{co}\,\mathcal{S}^*)}J(f)\le \max_{f\in\mathcal{U}} J(f),
$$
and therefore,
\begin{equation}
\max_{f\in\mathcal{U}} J(f)= \max_{f\in E(\overline{co}\,\mathcal{S}^*)}J(f).
\end{equation}

Another problem which has an independent interest in the theory of univalent functions is the estimation of the $L^p$ mean for certain classes of analytic functions. Corresponding to each analytic function $f$ in $\mathbb{D}$, we let
$$
J(f)=\frac{1}{2\pi}\int_{0}^{2\pi} |f^{(n)}(re^{i\theta})|^p\,d\theta,
$$
where $0<r<1$, $p>0$ and $n=0,1,2,\ldots$. It is our aim to maximize the functional $J(f)$ over the class $\mathcal{U}$. In general, it is more convenient to consider the functional $\|f\|=[J(f)]^{1/p}$. In particular, if $p\ge 1$ then $\|tf+(1-t)g\|\le t\|f\|+(1-t)\|g\|$ because of Minkowski's inequality. In other words, if $p\ge 1$ then $\|f\|$ is a convex functional. The above argument is due to MacGregor \cite{MacGregor-1972}.

\begin{thm}\label{p9-theorem-010}
If $f\in \mathcal{U}$ and $k(z)= z/(1-z)^2$ then
\begin{equation}\label{p9-025}
\frac{1}{2\pi}\int_{0}^{2\pi} |f^{(n)}(re^{i\theta})|^p\,d\theta \le \frac{1}{2\pi}\int_{0}^{2\pi} |k^{(n)}(re^{i\theta})|^p\,d\theta,
\end{equation}
whenever $0<r<1$, $p\ge 1$ and $n=0,1,2,\ldots$. Moreover, for $n=0$, the inequality (\ref{p9-025}) holds for any real number $p$.
\end{thm}

\begin{proof}
For $p\ge 1$, as mentioned above, it is sufficient to consider functions of the form
$$
f(z)=\frac{z}{(1-xz)^2},\quad |x|=1.
$$
For these functions $f(z)=x^{-1}k(xz)$ and hence $f^{(n)}(z)=x^{n-1}k^{(n)}(xz)$. Thus
\begin{align*}
\frac{1}{2\pi}\int_{0}^{2\pi} |f^{(n)}(re^{i\theta})|^p\,d\theta
&= \frac{1}{2\pi}\int_{0}^{2\pi} |k^{(n)}(xre^{i\theta})|^p\,d\theta\\[2mm]
&= \frac{1}{2\pi}\int_{0}^{2\pi} |k^{(n)}(re^{i\theta})|^p\,d\theta.
\end{align*}
Hence (\ref{p9-025}) holds for any $f\in \mathcal{U}$.

Moreover, if $f\in \mathcal{U}$ then, as mentioned above, $f(z)/z\prec k(z)/z$ and also $z/f(z)\prec z/k(z)$. Therefore, for $n=0$, the inequality (\ref{p9-025}) follows from Littlewood's subordination theorem (see \cite[Theorem 6.1]{Duren-book}) for every real $p$.
\end{proof}

For $f\in\mathcal{H}$ and $0<r<1$, let
$$
L(r,f)=\int_{0}^{2\pi} r|f'(re^{i\theta})|\,d\theta
$$
denote the arclength of the image of the circle $|z|=r$ under $f$. Substituting $n=1$ in Theorem \ref{p9-theorem-010}, we obtain the following sharp estimate of $L(r,f)$ for functions $f$ in $\mathcal{U}$.

\begin{cor}
If $f\in \mathcal{U}$ then $L(r,f)\le L(r,k)$, where $k$ is the Koebe function.
\end{cor}

We now review some of the standard facts on the theory of *-functions developed by Baernstein \cite{Baernstein-1974}. For more on *-functions we refer to Duren \cite{Duren-book}. For $g\in L^1[-\pi,\pi]$, the *-function of $g$ is defined by
$$
g^*(\theta)=\sup_{|E|=2\theta}\int_{E} g(x)\,dx, \quad 0\le\theta\le\pi,
$$
where $|E|$ denote the Lebesgue measure of the set $E$. Here the supremum is taken over all Lebesgue measurable subsets of $[-\pi,\pi]$ with $|E|= 2\theta$.

\begin{lem}\label{p9-lemma-001}
For $g,h\in L^1[-\pi,\pi]$, the following statements are equivalent.
\begin{enumerate}
\item[(a)] For every convex nondecreasing function $\Phi$ on $\mathbb{R}$,
$$
\int_{-\pi}^{\pi} \Phi(g(\theta))\,d\theta \le \int_{-\pi}^{\pi} \Phi(h(\theta))\,d\theta.
$$

\item[(b)] For every $t\in\mathbb{R}$,
$$
\int_{-\pi}^{\pi} [g(\theta)-t]^+\,d\theta \le \int_{-\pi}^{\pi} [h(\theta)-t]^+\,d\theta.
$$

\item[(c)] $g^*(\theta) \le h^*(\theta)$, $\quad 0\le\theta\le\pi$.

\end{enumerate}
Here $[g(\theta)-t]^+=\max\{g(\theta)-t,0\}$.
\end{lem}

\begin{lem}\cite[Lemma 2]{Leung-1979}\label{p9-lemma-005}
Let $u$ and $v$ be two subharmonic functions in $\mathbb{D}$ and there exists an analytic function $\omega: \mathbb{D} \rightarrow \mathbb{D}$ with $\omega(0)=0$ such that $u(z)= v(\omega(z))$ for $z\in\mathbb{D}$. Then for each $0<r<1$,
$$
u^*(re^{i\theta}) \le v^*(re^{i\theta}).
$$
\end{lem}

By using Lemmas \ref{p9-lemma-001} and \ref{p9-lemma-005} we prove the following interesting result.

\begin{thm}\label{p9-theorem-020}
Let $\Phi$ be a convex and nondecreasing function in $\mathbb{R}$. Then for $f\in \mathcal{U}(\lambda)$ with $0<\lambda\le 1$ and $0<r<1$, we have
$$
\int_{-\pi}^{\pi} \Phi(\pm \log|f(re^{i\theta})|)\,d\theta \le \int_{-\pi}^{\pi} \Phi(\pm \log|k_{\lambda}(re^{i\theta})|)\,d\theta
$$
where $k_{\lambda}$ is defined by (\ref{p9-010}).
\end{thm}

\begin{proof}
For $f\in \mathcal{U}(\lambda)$, it is known that $f(z)/z\prec k_{\lambda}(z)/z$. Since $k_{\lambda}(z)\ne 0$ in $\mathbb{D}\setminus\{0\}$, it follows that
\begin{equation}\label{p9-026}
\pm\log\frac{f(z)}{z}\prec \pm\log\frac{k_{\lambda}(z)}{z}.
\end{equation}
Since $\log|f(z)/z|$ is subharmonic in $\mathbb{D}$, an application of Lemma \ref{p9-lemma-005} in (\ref{p9-026}) gives
$$
\left(\pm\log\left|\frac{f(re^{i\theta})}{re^{i\theta}}\right|\right)^* \le \left(\pm\log\left|\frac{k_{\lambda}(re^{i\theta})}{re^{i\theta}}\right|\right)^*.
$$
For any $g\in L^1[-\pi,\pi]$ and a constant $c\in\mathbb{R}$, we note that $(g(re^{i\theta})+c)^*=g^*(re^{i\theta})+c$. Therefore,
\begin{equation}\label{p9-027}
\left(\pm\log|f(re^{i\theta})|\right)^* \le \left(\pm\log|k_{\lambda}(re^{i\theta})|\right)^*.
\end{equation}
The conclusion now follows from (\ref{p9-027}) together with Lemma \ref{p9-lemma-001}.
\end{proof}

\begin{cor}
If $f\in \mathcal{U}(\lambda)$ and $k_{\lambda}$ is defined by (\ref{p9-010}) then for any $0<r<1$, $p\in\mathbb{R}$, we have
\begin{equation*}\label{p9-025a}
\frac{1}{2\pi}\int_{0}^{2\pi} |f(re^{i\theta})|^p\,d\theta \le \frac{1}{2\pi}\int_{0}^{2\pi} |k_{\lambda}(re^{i\theta})|^p\,d\theta.
\end{equation*}
\end{cor}

\begin{proof}
For $p=0$, there is nothing to prove. For $p\ne0$, by choosing $\Phi(x)=e^{px}$, $p>0$, in Theorem \ref{p9-theorem-020}, we obtain the desired result.
\end{proof}

For a locally univalent function $f$ in $\mathbb{D}$, the pre-Schwarzian derivative $T_f$ is defined by $T_f=f''/f'$ and the pre-Schwarzian norm of $f$ is defined by
$$
\|f\|= \sup_{z\in\mathbb{D}} (1-|z|^2)|T_f(z)|.
$$
The pre-Schwarzian norm has significant meaning in the theory of Teichm\"{u}ller spaces. It is interesting to note that the pre-Schwarzian norm of $f$ is nothing but the Bloch seminorm of the function $\log f'$. In 1976, Yamashita \cite{Yamashita-1976} observed that $\|f\|$ is finite if, and only if, $f$ is uniformly locally univalent in $\mathbb{D}$, i.e. there exists a positive constant $\rho$ for which $f$ is univalent in every disk of hyperbolic radius $\rho$ in $\mathbb{D}$.
Furthermore, $\|f\|\leq 6$ if $f$ is univalent in $\mathbb{D}$ and conversely, $f$ is univalent in $\mathbb{D}$ if $\|f\|\leq 1$ and these bounds are sharp (see \cite{Becker-Pommerenke-1984}). Moreover, if $f$ can be extended to a $k$-quasiconformal automorphism of the Riemann sphere $\mathbb{\widehat{C}}$ then we have $\|f\|\leq 6k$. In 2004, Kim {\it et al.} \cite{Kim-Ponnusamy-Sugawa-2004b} obtained estimate of $\|J_\alpha[f]\|$ for functions in $\mathcal{U}(\lambda)$ with second coefficient $a_2$ fixed. Our next result gives estimate of $\|J_\alpha[f]\|$ for functions in $\mathcal{U}(\lambda)$.

\begin{thm}\label{p9-theorem-025}
For $f\in\mathcal{U}(\lambda)$, $0<\lambda\le1$, let $J_\alpha[f]$ be defined by (\ref{p9-005}) where $\alpha\in\mathbb{C}$. Then
\begin{equation}\label{p9-040}
\|J_\alpha[f]\| \le \|J_\alpha[k_\lambda]\|=
\begin{cases}
2|\alpha| & \text{for}\quad 0< \lambda \le \frac{1}{3}\\[3mm]
\frac{3+\lambda-2\sqrt{2(1-\lambda^2)}}{\lambda}|\alpha| & \text{for}\quad \frac{1}{3}< \lambda \le 1,
\end{cases}
\end{equation}
where $k_\lambda$ is defined by (\ref{p9-010}).
\end{thm}

\begin{proof}
Taking a logarithmic differentiation in (\ref{p9-005}), we obtain $J_\alpha[f]=\alpha J[f]$ and so
$$
\|J_\alpha[f]\|=|\alpha|\|J[f]\|.
$$
Hence it suffices to show the inequality (\ref{p9-040}) for the case $\alpha=1$. For $f\in\mathcal{U}(\lambda)$, let $F=J[f]$ and $K_\lambda=J[k_\lambda]$. Then by (\ref{p9-010}), we have
$$
F'(z)=\frac{f(z)}{z}\prec \frac{k_\lambda(z)}{z}=K'_\lambda(z).
$$
It is well-known that (see \cite[formula 4, page 35]{Pommerenke-book-1975}) if $g \prec h$ then
$$
\sup_{z\in\mathbb{D}}(1-|z|^2) |g'(z)|\le \sup_{z\in\mathbb{D}}(1-|z|^2) |h'(z)|.
$$
Using the fact $\log F'\prec \log K'_\lambda$, we obtain
$$
\|F\| = \sup_{z\in\mathbb{D}}(1-|z|^2)\left|\frac{F''(z)}{F'(z)}\right| \le \sup_{z\in\mathbb{D}}(1-|z|^2)\left|\frac{K''_\lambda(z)}{K'_\lambda(z)}\right|= \|K_\lambda\|.
$$
It remains to find the value of $\|J[k_\lambda]\|=\|K_\lambda\|$. Since $K'_\lambda(z)=k_\lambda(z)/z$, by taking logarithmic differentiation, we obtain
$$
\frac{K''_\lambda(z)}{K'_\lambda(z)} =\frac{1}{1-z}+\frac{\lambda}{1-\lambda z}
$$
and so
\begin{align*}
\sup_{z\in\mathbb{D}}(1-|z|^2)\left|\frac{K''_\lambda(z)}{K'_\lambda(z)}\right|
&= \sup_{z\in\mathbb{D}}(1-|z|^2)\left|\frac{1}{1-z}+\frac{\lambda}{1-\lambda z}\right|\\
&\le \sup_{z\in\mathbb{D}}(1-|z|^2)\left(\frac{1}{1-|z|}+\frac{\lambda}{1-\lambda |z|}\right)\\
&= \sup_{0<r<1}\phi(r),
\end{align*}
where
$$
\phi(r)= 1+r+\frac{\lambda(1-r^2)}{1-\lambda r}.
$$
It is important to note that for $z=t$ with $0<t<1$, we have
$$
(1-|z|^2)\left|\frac{K''_\lambda(z)}{K'_\lambda(z)}\right| = (1-t^2)\left|\frac{1}{1-t}+\frac{\lambda}{1-\lambda t}\right|=\phi(t).
$$
Therefore,
$$
\|J[k_\lambda]\|=\|K_\lambda\|=\sup_{0<r<1}\phi(r).
$$
To find the critical points of $\phi(r)$, we solve $\phi'(r)=0$. The roots of $\phi'(r)=0$ are given by
$$
r_0=\frac{2-\sqrt{2(1-\lambda^2)}}{\lambda} \quad\text{and}\quad r_1=\frac{2+\sqrt{2(1-\lambda^2)}}{\lambda}.
$$
But we note that $r_1>1$ for all $0<\lambda\le 1$, while $r_0<1$ for all $1/3<\lambda\le 1$. Moreover, it is an easy exercise to verify that
$$
\phi(r_0)=\frac{3+\lambda-2\sqrt{2(1-\lambda^2)}}{\lambda} \ge \phi(1) \quad\text{for}\quad \frac{1}{3}< \lambda \le 1.
$$
For $0<\lambda\le 1/3$, clearly $\phi(0)=1+\lambda$ and $\phi(1)=2$. Consequently,
$$
\|J[k_\lambda]\| = \sup_{0<r<1}\phi(r) =
\begin{cases}
2 & \text{for}\quad 0< \lambda \le \frac{1}{3}\\[3mm]
\frac{3+\lambda-2\sqrt{2(1-\lambda^2)}}{\lambda} & \text{for}\quad \frac{1}{3}< \lambda \le 1.
\end{cases}
$$
This completes the proof.
\end{proof}

\begin{cor}
For $f\in\mathcal{U}$, let $J_\alpha[f]$ be defined by (\ref{p9-005}) where $\alpha\in\mathbb{C}$. Then
$$
\|J_\alpha[f]\| \le \|J_\alpha[k]\|= 4|\alpha|
$$
where $k$ is the Koebe function.
\end{cor}

Our next result deals with the Fekete-Szeg\"{o} problem for functions in the class $\mathcal{U}(\lambda)$.

\begin{thm}\label{p9-theorem-030}
Let $f\in\mathcal{U}(\lambda)$, $0<\lambda\le1$ be of the form (\ref{p9-001}) and $\mu$ be a complex number. Then
\begin{equation}\label{p9-045}
|a_3-\mu a_2^2|\le
\begin{cases}
\displaystyle |(1+\lambda+\lambda^2)-\mu(1+\lambda)^2| & \text{for}\quad \left|\mu-\frac{1+\lambda+\lambda^2}{(1+\lambda)^2}\right|\ge \frac{1}{1+\lambda}\\[3mm]
\displaystyle  1+\lambda & \text{for}\quad \left|\mu-\frac{1+\lambda+\lambda^2}{(1+\lambda)^2}\right|\le \frac{1}{1+\lambda}.
\end{cases}
\end{equation}
The first inequality is sharp and equality holds for the function $k_\lambda$ defined by (\ref{p9-010}).
\end{thm}

\begin{proof}
If $f\in\mathcal{U}(\lambda)$ then from (\ref{p9-010}) there exists a function $\omega:\mathbb{D}\rightarrow \mathbb{D}$ with $\omega(0)$ and of the form $\omega(z)=\sum_{n=1}^{\infty} c_nz^n$ such that
\begin{equation}\label{p9-050}
\frac{f(z)}{z}=\frac{1}{(1-\omega(z))(1-\lambda\omega(z))}.
\end{equation}
In terms of series formulation, (\ref{p9-050}) can be written as
$$
\sum_{n=1}^{\infty} a_{n+1}z^n=\sum_{n=1}^{\infty} A_n(\omega(z))^n,
$$
where $A_n=1+\sum_{j=1}^{n}\lambda^j$. By comparing the coefficients of $z^n$ for $n=1,2$, we obtain
$$
a_2=(1+\lambda)c_1 \quad \text{and} \quad a_3=(1+\lambda)c_2+(1+\lambda+\lambda^2)c_1^2.
$$
Therefore,
\begin{align}\label{p9-055}
|a_3-\mu a_2^2|
&= \left|(1+\lambda)c_2+((1+\lambda+\lambda^2)-\mu(1+\lambda)^2)c_1^2 \right|\\[2mm]
&\le (1+\lambda)|c_2|+\left|(1+\lambda+\lambda^2)-\mu(1+\lambda)^2 \right||c_1|^2\nonumber\\[2mm]
&\le (1+\lambda)(1-|c_1|^2)+\left|(1+\lambda+\lambda^2)-\mu(1+\lambda)^2 \right||c_1|^2\nonumber\\[2mm]
&\le (1+\lambda)+\left\{\left|(1+\lambda+\lambda^2)-\mu(1+\lambda)^2 \right|-(1+\lambda)\right\}|c_1|^2.\nonumber
\end{align}
Since $|c_1|\le1$, the conclusion follows from the inequality (\ref{p9-055}).
\end{proof}

\begin{rem}
The first inequality of (\ref{p9-045}) is sharp and equality holds for the function $k_\lambda$ defined by (\ref{p9-010}), whereas the second inequality of (\ref{p9-045}) may not be sharp. For the function $g(z)=z/((1-z^2)(1-\lambda z^2))$, it is easy to see that $a_2(g)=0$ and $a_3(g)=1+\lambda$ and so $|a_3-\mu a_2^2|=1+\lambda$. But
$$
\left|g'(z)\left(\frac{z}{g(z)}\right)^2-1 \right|= \left|z^2(1+\lambda-3\lambda z^2) \right|>1
$$
when $z$ is purely imaginary number. This shows that $g\not\in\mathcal{U}(\lambda)$.
\end{rem}


If $f\in\mathcal{U}(\lambda)$ is of the form (\ref{p9-001}) then a simple exercise shows that
$$
f'(z)\left(\frac{z}{f(z)}\right)^2-1= (a_3-a_2^2)z^2+\cdots
$$
and so by the Schwarz lemma we obtain $|a_3-a_2^2|\le \lambda$. Thus, for $\mu\ge1$ we immediately have
\begin{align}\label{p9-060}
|a_3-\mu a_2^2|
&\le |a_3- a_2^2|+(\mu-1)|a_2|^2\\[2mm]
&\le \lambda+(\mu-1)(1+\lambda)^2\nonumber\\[2mm]
&=\mu(1+\lambda)^2-(1+\lambda+\lambda^2).\nonumber
\end{align}
The equality in (\ref{p9-060}) holds for the function $k_\lambda$. If we consider $\mu$ as a real number in Theorem \ref{p9-theorem-030} then the condition
$$
\left|\mu-\frac{1+\lambda+\lambda^2}{(1+\lambda)^2}\right|\ge \frac{1}{1+\lambda}
$$
is equivalent to
$$
\mu\in\left(-\infty, \left(\lambda/(1+\lambda)\right)^2\right) \cup \left(1+1/(1+\lambda)^2,\infty\right).
$$
From the above discussion, we immediately have

\begin{cor}
Let $f\in\mathcal{U}(\lambda)$ be of the form (\ref{p9-001}) and $\mu$ be a real number. Then the following sharp inequality
$$
|a_3-\mu a_2^2|\le |(1+\lambda+\lambda^2)-\mu(1+\lambda)^2|\quad\text{for}\quad \mu\in\left(-\infty, \left(\lambda/(1+\lambda)\right)^2\right) \cup \left(1,\infty\right).
$$
holds and equality attains for the function $k_\lambda$.
\end{cor}

\section{Properties of meromorphic functions associated with $\mathcal{U}(\lambda)$}\label{Main Results-2}

Let $\Sigma$ be the class of meromorphic and univalent functions $g$ on $\Delta:=\{\zeta\in\mathbb{\widehat{C}}:|\zeta|>1\}$ of the form
\begin{equation}\label{p9-030}
g(\zeta)=\zeta+\sum_{n=0}^{\infty} b_n\zeta^{-n}, \quad 1<|\zeta|<\infty.
\end{equation}
It is well recognized that the set $\Sigma$ plays an important role in the study of the class $\mathcal{S}$. The map $f\mapsto g(\zeta)=1/f(1/\zeta)$ gives a one-to-one corresponds between the class $\mathcal{S}$ and the class $\{g\in\Sigma:g(z)\ne 0\}$. Moreover, we have the following formula
$$
g'(\zeta)=\left(\frac{z}{f(z)}\right)^2f'(z),\quad \zeta=\frac{1}{z},
$$
and hence, functions $f$ in $\mathcal{U}(\lambda)$ are associated with the functions $g$ in $\mathcal{M}(\lambda):=\{g\in\Sigma:|g'(\zeta)-1|<\lambda\}$. Further, let
$$
\Sigma_{0}:=\{g\in\Sigma: b_0(g)=0\} \quad\mbox{ and }\quad \mathcal{M}_{0}(\lambda):=\{g\in\mathcal{M}(\lambda): b_0(g)=0\}.
$$
It is important to note that the classes $\Sigma$ and $\mathcal{M}(\lambda)$ are not compact. However, the classes $\Sigma_{0}$ and $\mathcal{M}_{0}(\lambda)$ are compact with respect to the topology of uniform convergence. The classes $\mathcal{M}(\lambda)$ and $\mathcal{M}_{0}(\lambda)$ have been extensively studied by Vasudevarao and Yanagihara \cite{Vasudevarao-Yanagihara-2013}. For $g\in\Sigma$, let $E(g)$ be the omitted set of $g$, i.e., $E(g)=\mathbb{\widehat{C}}\setminus g(\Delta)$. In \cite{Vasudevarao-Yanagihara-2013}, it has been proved that for $0<\lambda \le 1$ and $g\in\mathcal{M}(\lambda)$,
$$
\pi(1-\lambda^2)\le {\rm area\,} (E(g))=\pi\left(1-\sum_{n=1}^{\infty} n|b_n|^2\right) \le \pi.
$$
From this it can be seen that
\begin{equation}\label{p9-035}
\sum_{n=1}^{\infty} n|b_n|^2 \le \lambda^2
\end{equation}
and equality holds in (\ref{p9-035}) for a function $g\in\mathcal{M}(\lambda)$ if, and only if, ${\rm area\,} (E(g))=\pi(1-\lambda^2)$.

In 1955, Springer \cite{Springer-1955} proved that the set of extreme points of $\Sigma_{0}$ contains all functions $g\in\Sigma_{0}$ with ${\rm area\,} (E(g))=0$. We note that the classes $\Sigma$ and $\mathcal{M}(\lambda)$ have no extreme points. Our next result gives a sufficient condition for a function $g$ to be an extreme point of $\mathcal{M}_{0}(\lambda)$.

\begin{thm}\label{p9-theorem-035}
Let $g\in\mathcal{M}_{0}(\lambda)$ with ${\rm area\,} (E(g))=\pi(1-\lambda^2)$. Then $g$ is an extreme point of $\mathcal{M}_{0}(\lambda)$.
\end{thm}

\begin{proof}
Suppose that a function $g\in\mathcal{M}_{0}(\lambda)$ with ${\rm area\,} (E(g))=\pi(1-\lambda^2)$ has a representation $g=tg_1+(1-t)g_2$, $0<t<1$, as a proper convex combination of two functions $g_1,g_2\in\mathcal{M}_{0}(\lambda)$. Let $g(\zeta)$ be of the form (\ref{p9-030}) and $g_1(z)$ and $g_2(z)$ have the representation of the form $g_1(\zeta)=\zeta+\sum_{n=0}^{\infty} c_n\zeta^{-n}$ and $g_2(\zeta)=\zeta+\sum_{n=0}^{\infty} d_n\zeta^{-n}$. Then the case of equality in (\ref{p9-035}) gives
\begin{align*}
\lambda^2 &= \sum_{n=1}^{\infty} n|b_n|^2 =\sum_{n=1}^{\infty} n|tc_n+(1-t)d_n|^2\\
&= \sum_{n=1}^{\infty} n\left\{t|c_n|^2+(1-t)|d_n|^2-t(1-t)|c_n-d_n|^2\right\}\\
&\le t\lambda^2+(1-t)\lambda^2-t(1-t)\sum_{n=1}^{\infty} n|c_n-d_n|^2.
\end{align*}
Since $0<t<1$, it follows that
$$
\sum_{n=1}^{\infty} n|c_n-d_n|^2 \le 0
$$
which implies $c_n=d_n$ for $n=1,2,\ldots$. In other words, $g_1=g_2$ and hence, $g$ is an extreme point of $\mathcal{M}_{0}(\lambda)$.
\end{proof}
\noindent\textbf{Acknowledgement:}
The second author thank SERB, Govt of India for the support.

\end{document}